\documentclass[11pt]{article}
\usepackage{amsmath, amssymb, amsthm}

\date{}
\title{\vspace{-0.8cm}Decomposing random graphs into few cycles and edges}
\author{
D\'aniel Kor\'andi \thanks{Department of Mathematics, ETH, 8092 Zurich. Email: daniel.korandi@math.ethz.ch.}
\and
Michael Krivelevich \thanks{School of Mathematical Sciences, Raymond and Beverly Sackler Faculty of Exact Sciences, Tel Aviv University,
Tel Aviv 6997801, Israel. Email: krivelev@post.tau.ac.il. Research supported in part by a USA-Israel BSF grant and by a grant from the Israel Science
Foundation.}
\and
Benny Sudakov \thanks{Department of Mathematics, ETH, 8092 Zurich.
Email: benjamin.sudakov@math.ethz.ch.
Research supported in part by SNSF grant 200021-149111 and
by a USA-Israel BSF grant.}
}

\theoremstyle{plain}
\newtheorem{THM}{Theorem}[section]
\newtheorem*{THM*}{Theorem}
\newtheorem{PROP}[THM]{Proposition}
\newtheorem{LEMMA}[THM]{Lemma}

\newtheorem{COR}[THM]{Corollary}
\newtheorem{CLAIM}[THM]{Claim}

\theoremstyle{definition}

\newcommand{\Prb}{\mathbb{P}}
\newcommand{\Exp}{\mathbb{E}}
\newcommand{\floor}[1]{\left\lfloor #1 \right\rfloor}
\newcommand{\ceil}[1]{\left\lceil #1 \right\rceil}

\newcommand{\om}[1]{\omega\left( #1 \right)}
\newcommand{\subs}{\subseteq}
\newcommand{\eps}{\varepsilon}
\newcommand{\Bin}{\textrm{Bin}}
\newcommand{\odd}{\textrm{odd}}

\begin{document}
\maketitle

\begin{abstract}
Over 50 years ago, Erd\H{o}s and Gallai conjectured that the edges of every graph on $n$ vertices can be decomposed into  $O(n)$ cycles and edges. Among other results, Conlon, Fox and Sudakov recently proved that this holds for the random graph $G(n,p)$ with probability approaching 1 as $n\rightarrow\infty$. In this paper we show that for most edge probabilities $G(n,p)$ can be decomposed into a union of $\frac{n}{4}+\frac{np}{2}+o(n)$ cycles and edges whp.
This result is asymptotically tight.
\end{abstract}


\section{Introduction}

Problems about packing and covering the edge set of a graph using cycles and paths have been intensively studied since the 1960s. One of the oldest questions in this area was asked by Erd\H{o}s and Gallai \cite{E83,EGP66}. They conjectured that the edge set of any graph $G$ on $n$ vertices can be covered by $n-1$ cycles and edges, and can be partitioned into a union of $O(n)$ cycles and edges. The covering part was proved by Pyber \cite{P85} in the 1980s, but the partitioning part is still open. As noted in \cite{EGP66}, it is not hard to show that $O(n\log n)$ cycles end edges suffice. This bound was recently improved to $O(n\log\log n)$ by Conlon, Fox and Sudakov in \cite{CFS14}, where they also proved that the conjecture holds for random graphs and graphs of linear minimum degree. The present paper treats the problem in the case of random graphs in more detail.

Let $0\le p(n)\le 1$ and define $G(n,p)$ to be the random graph on $n$ vertices, where the edges are included independently with probability $p$. We hope to find a close to optimal partition of the edges of a random graph into cycles and edges. Observe that in any such partition each odd-degree vertex needs to be incident to at least one edge, so if $G(n,p)$ has $s$ odd-degree vertices, then we certainly need at least $s/2$ edges. Also, a typical random graph has about $\binom{n}{2}p$ edges, whereas a cycle may contain no more than $n$ edges, so we need at least about $\frac{np}{2}$ cycles (or edges) to cover the remaining edges. This simple argument gives a lower bound of $\frac{np}{2}+\frac{s}{2}$ on the optimum.

In this paper we show that this lower bound is asymptotically tight. Let $\odd(G)$ denote the number of odd-degree vertices in the graph $G$. We say that $G(n,p)$ (with $p=p(n)$) satisfies some property $P$ \emph{with high probability} or \emph{whp} if the probability that $P$ holds tends to 1 as $n$ approaches infinity. Our main result is the following theorem.

\begin{THM} \label{thm:main}
Let the edge probability $p(n)$ satisfy $p=\om{ \frac{\log\log n}{n} }$. Then whp $G(n,p)$ can be split into $\frac{\odd(G(n,p))}{2}+\frac{np}{2}+o(n)$ cycles and edges.
\end{THM}

In fact, as we show in Lemma~\ref{thm:odd_degrees}, for most of the $p$'s in the range, $\odd(G(n,p))\sim \frac{n}{2}$. This immediately implies the following, perhaps more tangible, corollary.

\begin{COR} \label{cor:main}
Let $p=p(n)$ be in the range $[\om{ \frac{\log\log n}{n} }, 1-\om{ \frac{1}{n} }]$. Then whp $G(n,p)$ can be split into $\frac{n}{4}+\frac{np}{2}+o(n)$ cycles and edges.
\end{COR}

Here we use the standard notation of $\om{f}$ for any function that is asymptotically greater than the function $f(n)$, i.e., $g(n)=\om{f(n)}$ if  $\lim_{n\rightarrow \infty} \frac{g(n)}{f(n)}=\infty$. In this paper $\log$ stands for the natural logarithm, and for the sake of clarity we omit the floor and ceiling signs whenever they are not essential. We call $G$ an Euler graph if all the vertices of $G$ have even degree (not requiring that $G$ be connected).

\subsection{Proof outline}

We will break the probability range into three parts (the sparse, the intermediate and the dense ranges), and prove Theorem~\ref{thm:main} separately for each part. The proofs of the denser cases build on the sparser cases, but all the proofs have the following pattern: we start with deleting $\frac{\odd(G(n,p))}{2}+o(n)$ edges so that the remaining graph is Euler, and then we extract relatively few long cycles to reduce the problem to a sparser case.

In the sparse case we will check the Tutte condition to show that there is a large matching on the odd-degree vertices, and then use expansion properties to iteratively find cycles that are much longer than the average degree. In the end, we are left with a sparse Euler subgraph, which breaks into $o(n)$ cycles.

The denser cases are somewhat more complicated. We will need to break $G(n,p)$ into several random subgraphs. These graphs will not be independent, but we can remove edges from one of them without affecting the random structure of the others.

In the intermediate case we break $G(n,p)$ into three random subgraphs, $G(n,p)=G_1\cup G_2\cup G_3$. First we find an edge set $E_0$ in $G_2$ such that $G(n,p)-E_0$ is Euler. Then we break $(G_2\cup G_3) - E_0$ into matchings and $G_1$ into even sparser random graphs. Using a result by Broder, Frieze, Suen and Upfal \cite{BFSU} about paths connecting a prescribed set of vertex pairs in random graphs, we connect the matchings into cycles using the parts of $G_1$. The remaining edges are all from $G_1$, and the tools from the sparse case take care of them.

In the dense case we break into four subgraphs, $G(n,p)=G_1\cup G_2\cup G_3\cup G_4$, where $G_4$ contains the majority of the edges. Again we start by finding the edge set $E_0$ in $G_3$. Next, we apply a recent packing result by Knox, K\"uhn and Osthus \cite{KKO} to find many edge-disjoint Hamilton cycles in $G_4$. Then we break the remaining edges from $G_3\cup G_4$ into matchings and use $G_2$ to connect them into cycles. At this point we have a still intact random graph $G_1$ and some edges from $G_2$ left, and these fit into the intermediate setting, hence the previous results complete the proof.

The paper is organized as follows: in Section~\ref{sec:odd} we prove all the results we need about odd-degree vertices, including the typical existence of $E_0$ and the fact that normally about half the vertices are odd. Section~\ref{sec:sparse} shows why we can iteratively remove relatively long cycles from sparse random graphs, and proves Theorem~\ref{thm:main} in the sparse range. In Section~\ref{sec:lemma} we show the details of breaking into matchings and then connecting them into few cycles, and prove the main lemma for the intermediate range. Finally, we complete the proofs of the denser cases in Section~\ref{sec:dense}.

\section{Covering the odd-degree vertices} \label{sec:odd}

The first step in the proof is to reduce the problem to an Euler subgraph by removing relatively few edges. This section is devoted to the discussion of the results we need about odd-degree vertices in random graphs.

Throughout the paper, we will use the following Chernoff-type inequalities (see, e.g., \cite{ASBOOK}).

\begin{CLAIM} \label{lem:chernoff}
Suppose we have independent indicator random variables $X_1, \ldots, X_n$, where $X_i=1$ with probability $p_i$ and $X_i=0$ otherwise, for all $i=1, \ldots, n$. Let $X=X_1+\cdots+X_n$ be the sum of the variables and $p=(p_1+\cdots+p_n)/n$ be the average probability so that $\Exp(X)=np$. Then:
\begin{itemize}
  \item[(a)] $\Prb(X > np+a) \le e^{-2a^2/n}$ for all $a > 0$,
  \item[(b)] $\Prb(X > 2np) \le e^{-np/20}$,
  \item[(c)] $\Prb(X < np-a) \le e^{-a^2/np}$ and hence
  \item[(d)] $\Prb(X < np/2) \le e^{-np/20}.$
\end{itemize}
In particular, all the above estimates hold when $X\sim \Bin (n,p)$ is a binomial random variable.
\end{CLAIM}

In the coming results, we will also make use of the following easy observation: For $X\sim \Bin (n,p)$, the probability that $X$ is odd is exactly
\[ \sum_{i=0}^{\floor{\frac{n-1}{2}}} \binom{n}{2i+1}p^{2i+1}(1-p)^{n-(2i+1)} = \frac{(1-p+p)^n - (1-p-p)^n}{2}= \frac{1-(1-2p)^n}{2}. \]

First we prove the following estimate on the number of vertices of odd degree in $G(n,p)$.

\begin{CLAIM} \label{thm:odd_degrees}
Suppose the probability $p$ satisfies $\om{ \frac{1}{n} } < p < 1-\om{ \frac{1}{n} }$. Then whp $G\sim G(n,p)$ contains $\frac{n}{2} (1+o(1))$ vertices of odd degree.
\end{CLAIM}
\begin{proof}
Let us fix a bipartition of the vertices into two nearly equal parts $V=V_1\cup V_2$, where $n_1=|V_1|=\floor{n/2}$ and $n_2=|V_2|=\ceil{n/2}$. Our plan is to show that whp roughly half the vertices of $V_1$ and roughly half the vertices of $V_2$ have odd degree in $G$.

So let $Y_1$ be the number of odd-degree vertices in $V_1$. The probability that a specific vertex $v$ has odd degree is $\frac{1-(1-2p)^{n-1}}{2}$, so by the linearity of expectation, the expected number of odd degree vertices in $V_1$ is
\[ \Exp(Y_1)=\frac{n_1}{2}-\frac{n_1(1-2p)^{n-1}}{2}\sim \frac{n_1}{2} \]
for $\om{1/n}\le p \le 1-\om{1/n}$. We still need to prove that $Y_1$ is tightly concentrated around its mean.

Let us now expose the random edges spanned by $V_1$. Then the degrees of the vertices in $V_1$ are determined by the crossing edges between $V_1$ and $V_2$. At this point, each vertex in $V_1$ has $n_2$ incident edges unexposed, and these edges are all different. So let $v$ be any vertex in $V_1$. The probability that $v$ is connected to an odd number of vertices in $V_2$ is $p_1=\frac{1-(1-2p)^{n_2}}{2} \sim \frac{1}{2}$, so the probability that $v$ will end up having an odd degree is either $p_1$ or $(1-p_1)$ (depending on the parity of its degree in $G[V_1]$). This means that, conditioning on any collection of edges spanned by $V_1$, $Y_1$ is the sum of $n_1$ indicator variables with probabilities $p_1$ or $1-p_1$, so we can apply Claim~\ref{lem:chernoff} (a) and (c) with average probability $\min\{p_1,1-p_1\} \leq \bar{p}_1 \leq \max\{p_1,1-p_1\}$ to get
\[ \Prb(|Y_1-\Exp(Y_1)| > a) \le e^{-2a^2/n_1} + e^{-a^2/n_1\bar{p}_1}. \]
Then taking $a=n^{2/3}$ and using that $\bar{p}_1 \sim 1/2$ we get
\[ \Prb(|Y_1-\Exp(Y_1)| > n^{2/3}) \le 3e^{-2n^{1/3}} =o(1). \]

Repeating the same argument for $V_2$, we see that if $Y_2$ is the number of odd-degree vertices in $V_2$, then
\[ \Exp(Y_2)=\frac{n_2-n_2(1-2p)^{n-1}}{2} \sim n_2/2 \]
and
\[ \Prb(|Y_2-\Exp(Y_2)| > n^{2/3}) \le 3e^{-2n^{1/3}}. \]

So with probability at least $1-6e^{-2n^{1/3}} = 1-o(1)$ the number of odd vertices is $\frac{n}{2}(1+o(1))$.
\end{proof}

In order to show that typically there is a small set of edges covering the odd vertices, we will need some properties of sparse random graphs. The following somewhat technical lemma collects the tools we use in the proof.

\begin{LEMMA} \label{lem:prep}
Let $p=p(n)$ satisfy  $\om{ \frac{\log\log n}{n} } \le p \le \frac{\log^{10} n}{n}$. Then the following properties hold whp for $G\sim G(n,p)$:
\begin{enumerate}
  \item the giant component of $G$ covers all but at most $o(n)$ vertices, and all other components are trees,
  \item the independence number $\alpha(G)$ is at most $\frac{2\log (np)}{p}$,
  \item the diameter of the giant component of $G$ is at most $\frac{2\log n}{\log (np)}$,
  \item any set $T$ of at most $2n^{1/10}$ vertices spans at most $2|T|$ edges,
  \item any two disjoint vertex sets $T, T'$ of the same size $n^{1/10} \le t \le n/100$ are connected by less than $tnp/6$ edges, and
  \item all but at most $n/2\log^2 n$ vertices in $G$ have at least $np/5$ odd-degree neighbors.
\end{enumerate}
\end{LEMMA}
\begin{proof}
For the first two properties we refer to Bollob\'as \cite{BBOOK}, while the third property was proved by Chung and Lu \cite{CL}. To prove the fourth one, note that the probability that a fixed set $T$ of size $t$ spans at least $2t$ edges is at most $\binom{\binom{t}{2}}{2t} \cdot p^{2t}$. So the probability that some $T$ of size $t\le 2n^{1/10}$ spans at least $2t$ edges is at most
\[ \sum_{t=5}^{2n^{1/10}} \binom{n}{t}\cdot \binom{t^2}{2t}\cdot p^{2t} \le \sum_{t=5}^{2n^{1/10}}( nt^4p^2 )^t \le \sum_{t=5}^{2n^{1/10}} n^{-t/2} = o(1). \]

We use a similar argument to prove the fifth property. For fixed $T_1$ and $T_2$ of size $t$, the probability that at least $tnp/6$ edges connect them in $G$ is at most $\binom{t^2}{tnp/6}p^{tnp/6}$. So the probability that the property does not hold can be bounded from above by
\[ \sum_{t=n^{1/10}}^{n/100} \binom{n}{t}^2\cdot \binom{t^2}{tnp/6}p^{tnp/6} \le \sum_{t=n^{1/10}}^{n/100} \left( \frac{en}{t}\right)^{2t}\cdot \left( \frac{et^2p}{tnp/6} \right)^{tnp/6} \le \sum_{t=n^{1/10}}^{n/100} \left( \frac{e^2n^2}{t^2} \cdot \left( \frac{6et}{n} \right)^{\log\log n/6}  \right)^t. \]
Here $\frac{6et}{n}<\frac{1}{2}$, so for large enough $n$ this probability is smaller than $\sum_{t=n^{1/10}}^{n/100} (1/2)^t < 2\cdot 2^{-n^{1/10}}=o(1)$.

The last property has a similar flavor to Claim~\ref{thm:odd_degrees}. We will prove that the probability that a particular vertex has fewer than $np/5$ odd neighbors is at most $1/\log^3 n$. Then the expected number of such vertices is at most $n/\log^3 n$, hence the probability that there are more than $n/2\log^2 n$ of them is, by Markov's inequality, at most $2/\log n =o(1)$.

So let us pick a vertex $v$ in $G$. Using Claim~\ref{lem:chernoff}, we see that the probability that it has fewer than $np/2$ or more than $2np$ neighbors is at most $2e^{-np/20}$. Assume this is not the case, and expose the edges spanned by the neighborhood $N$ of $v$. Now the number of odd neighbors of $v$ is determined by the edges between $N$ and $\overline{N}=V-(N\cup v)$. In fact, the probability that a vertex $u\in N$ is connected to an odd number of vertices in $\overline{N}$ is $p_1=\frac{1-(1-2p)^{n-d-1}}{2} \sim \frac{1}{2}$, where $d=|N|<n/2$ is the degree of $v$, so $u$ has an odd degree in $G$ with probability $p_1$ or $(1-p_1)$. Thus the number of odd neighbors of $v$ is a sum of $d$ indicator random variables with probabilities $p_1$ or $(1-p_1)$. Another application of Claim~\ref{lem:chernoff} then shows that the probability that $v$ has fewer than $np/5 < dp_1/2$ odd neighbors (conditioned on $np/2\le d\le 2np$) is at most $e^{-dp_1/20}<e^{-np/50}$.

Summarizing the previous paragraph, the probability that $v$ has fewer than $np/5$ odd neighbors is at most
\[  2e^{-np/20}+e^{-np/50} \le 3e^{-np/50} \le e^{-3\log\log n}=\frac{1}{\log^3 n}  \]
for large enough $n$, establishing the sixth property.
\end{proof}

Now we are ready to prove the following statement, which will serve as a tool to get rid of odd degrees. We should point out that this lemma only works for $p\gg \frac{\log n}{n}$. However, its flexibility -- the fact that we can apply it to any set $S$ -- will prove useful when tackling the dense case.

\begin{LEMMA} \label{lem:matching}
Let $p=\om{ \frac{\log n}{n} }$ but $p\le \frac{\log^{10} n}{n}$. Then whp in $G\sim G(n,p)$ for any vertex set $S$ of even cardinality, there is a collection $E_0$ of $\frac{|S|}{2}+o(n)$ edges in $G$ such that $S$ is exactly the set of vertices incident to an odd number of edges in $E_0$.
\end{LEMMA}
\begin{proof}
Take any set $S$. First, we find a matching in $S_0=S$ greedily, by selecting one edge $e_i$ at a time spanned by $S_i$, and then removing the vertices of $e_i$ from $S_i$ to get $S_{i+1}$. At the end of the process, the remaining set of vertices $S'\subs S$ is independent in $G$. The second property from Lemma~\ref{lem:prep} implies that $|S'|\le \frac{2\log(np)}{p}$. Now let us pair up the vertices in $S'$ arbitrarily, and for each of these pairs $\{v_{j,1},v_{j,2}\}$ take a shortest path $P_j$ in $G$ connecting them. (Recall that for $p=\om{ \frac{\log n}{n} }$ the random graph $G(n,p)$ is whp connected.) The third property ensures that each of the $P_j$ contains at most $\frac{2\log n}{\log (np)}$ edges. Note that we do not assume these paths to be edge-disjoint and they may contain the $e_i$'s as well.

Let us define $E_0$ to be the ``mod 2 union'' of the $P_j$ and the $e_i$, i.e., we include an edge $e$ in $E_0$ if it appears an odd number of times among them. Then the set of odd-degree vertices in $E_0$ is indeed $S$, and the number of edges is at most
\[ \frac{|S|}{2}+\frac{2\log (np)}{p}\cdot \frac{2\log n}{\log (np)}= \frac{|S|}{2}+o(n) \]
since $np\gg \log n$.
\end{proof}

We can actually push the probability $p$ down a bit by giving up on the above mentioned flexibility. The following lemma takes care of the odd-degree vertices and the small components in the sparse case.

\begin{LEMMA} \label{lem:matching2}
Let $p=\om{ \frac{\log\log n}{n} }$, but $p \le \frac{\log^ {10}n}{n}$, and let $S$ be the set of odd-degree vertices in $G\sim G(n,p)$. Then whp there is a collection $E_0$  of $\frac{|S|}{2}+o(n)$ edges in $G$ such that $G-E_0$ is an Euler graph.
\end{LEMMA}
\begin{proof}
Let us assume that $G$ satisfies all the properties from Lemma~\ref{lem:prep}. Our plan is to show that there is a matching in the induced subgraph $G_S=G[S]$ covering all but at most $n/\log^2 n$ vertices in $S$, using the defect version of Tutte's theorem on $G_S$. For this we need that the deletion of any set $T$ of $t$ vertices from $G_S$ creates no more than $t+n/\log^2 n$ odd-vertex components. In fact, we will prove that the deletion of $t$ vertices breaks $G_S$ into at most $t+n/\log^2 n$ components.

If $t\ge \frac{n}{100}$ then this easily follows from the second property: if at least $t$ components were created, then we could find an independent set in $G$ of size $t$ simply by picking a vertex from each component. But $\alpha(G)\le\frac{2\log (np)}{p}< \frac{n}{100}$, so this is impossible.

Now suppose there is a set $T$ of $t< \frac{n}{100}$ vertices such that $G_S-T$ has at least $t+\frac{n}{\log^2 n}$ components. Here the number of components containing at least $2\log^2 n$ vertices is clearly at most $\frac{n}{2\log^2 n}$. On the other hand, according to the sixth property of Lemma ~\ref{lem:prep}, there are at most $\frac{n}{2\log^2 n}$ components containing a vertex of degree less than $np/5$ in $G_S$. Thus there are $t$ components of size at most $2\log^2 n$, and hence of average degree at most 4 by the fourth property, such that all the vertices contained in them have at least $np/5$ neighbors in $G_S$. Each such component then has a vertex with at least $np/5-4$ neighbors in $T$. Pick a vertex like that from $t$ of these components to form the set $T'$.

We see that there are at least $t(np/5-4) > tnp/6$ edges going between $T$ and $T'$, but this contradicts the fourth property if $t\le n^{1/10}$ and the fifth property if $n^{1/10} < t < \frac{n}{100}$. So by Tutte's theorem, we can find some edge set $M$ that forms a matching in $S$, covering all but $\frac{n}{\log^2 n}$ of its vertices.

Now let $F$ be the set of edges appearing in the small components of $G$ (recall that in this probability range, $G(n,p)$ whp has a giant component and possibly some smaller ones). According to the first property, these edges span trees covering $o(n)$ vertices in total, so $|F|=o(n)$. We see that the set $M\cup F$ takes care of all odd vertices outside the giant component and all but at most $\frac{n}{\log^2 n}$ of them inside. The rest of the proof follows the idea from the previous lemma: we pair up the remaining odd vertices arbitrarily and take shortest paths $P_j$ in $G$ connecting them. Once again, the third property implies that each path has length at most $\frac{2\log n}{\log (np)}$. Then $E_0$, the ``mod 2 union'' of the $P_j$ and $M\cup F$ satisfies all requirements and uses at most
\[ \frac{|S|}{2} + o(n) + \frac{n}{\log^2 n}\cdot \frac{2\log n}{\log (np)}= \frac{|S|}{2}+o(n) \]
edges.
\end{proof}

\section{Cycle decompositions in sparse random graphs} \label{sec:sparse}

In this section we show that Euler subgraphs of sparse random graphs can be decomposed into $o(n)$ edge-disjoint cycles. The following statement from \cite{CFS14}, which we use to find long cycles, is an immediate consequence of applying P\'osa's rotation-extension technique \cite{P76} as described in \cite{BBDK}.

\begin{LEMMA} \label{lem:posa}
If a graph $G$ does not contain any cycle of length at least $3t$, then there is a set $T$ of at most $t$ vertices such that $|N(T)|\le 2|T|$.
\end{LEMMA}

We say that a graph $G$ is \emph{sufficiently sparse} if any set of vertices $S$ spans less than $r|S|$ edges, where $r=\max\{\frac{|S|}{12\log^2 n},7 \}$. Note that any subgraph of a sufficiently sparse graph is also sufficiently sparse.

In what follows, we proceed by showing that on the one hand, for $p$ small enough the graph $G(n,p)$ is typically sufficiently sparse, while on the other hand, any sufficiently sparse graph contains few edge-disjoint cycles covering most of the edges.

\begin{LEMMA} \label{lem:sparseg}
Let $p=p(n)<n^{-1/6}$ and let $G\sim G(n,p)$. Then whp $G$ is sufficiently sparse.
\end{LEMMA}
\begin{proof}
For fixed $s$, the probability that there is an $S$ of size $s$ containing at least $rs$ edges is at most
\[ \binom{n}{s}\cdot \binom{\binom{s}{2}}{rs}\cdot p^{rs} \le n^s \cdot \left(\frac{esp}{2r}\right)^{rs}= \left(n\cdot \left(\frac{esp}{2r}\right)^r\right )^s. \]
Put $x=n\cdot \left(\frac{esp}{2r}\right)^r$. If we further assume that $r\ge\frac{s}{12\log^2 n}$ and $r\ge7$ then we get that for large $n$ and $p<n^{-1/6}$
\[ x \le n\cdot (6ep\log^2 n)^r \le n\cdot (6e)^7 \left( \frac{\log^2 n}{n^{1/6}} \right)^7 = o(1),\]
where we used the fact that $6ep\log^2 n<1$ for large $n$. Hence the probability that for some $s$ there is a set $S$ of size $s$ which spans more than $\max\{\frac{s}{12\log^2 n}, 7\}\cdot s$ edges, i.e., that  $G$ is not sufficiently sparse, is at most
\[ \sum_{i=2}^{n} x^i < \frac{x^2}{1-x} < x =o(1).\]
\end{proof}

\begin{COR} \label{lem:sparse}
For $p<n^{-1/6}$, whp all subgraphs of $G(n,p)$ are sufficiently sparse.
\end{COR}

Using these lemmas we can derive one of our main tools, the fact that we can make a subgraph of a random graph relatively sparse by iteratively removing long cycles.

\begin{PROP} \label{prop:cycrem}
Let $H$ be a sufficiently sparse $n$-vertex graph of average degree $d>84$. Then it contains a cycle of length at least $d\log^2 n$.
\end{PROP}
\begin{proof}
Assume to the contrary that there is no such cycle. Define $H\subs G$ to be the $d/2$-core of $G$, i.e., the non-empty subgraph obtained by repeatedly removing vertices of degree less than $d/2$. Since there is no cycle of length at least $d\log^2 n$ in $H$, we can apply Lemma~\ref{lem:posa} to find a set $T$ of $t\le d\log^2 n/3$ vertices such that $|N(T)|\le 2|T|=2t$. Let $S=T\cup N(T)$ and $s=|S|$, then the minimum degree condition implies that there are at least $dt/4$ edges incident to $T$, hence the set $S$ spans at least $sd/12$ edges.

Note that $r=d/12>7$. Also, since $s\le 3t\le d\log^2 n$, we have $r\ge s/12\log^2 n$. So the set $S$ spans at least $\max\{\frac{s}{12\log^2 n},7\}\cdot |S|$ edges, contradicting the assumption that $H$ is sufficiently sparse.
\end{proof}

\begin{COR} \label{thm:cycrem}
Let $p<n^{-1/6}$. Then whp $G\sim G(n,p)$ has the following property. If $H$ is a subgraph of $G$ on $n_0$ vertices of average degree $d$, then $H$ can be decomposed into a union of at most $2n_0/\log n_0$ cycles and a graph $H'$ of average degree at most $84$.
\end{COR}
\begin{proof}
By Corollary~\ref{lem:sparse} all subgraphs of $G(n,p)$ are sufficiently sparse whp. So we can repeatedly apply Proposition~\ref{prop:cycrem} to remove cycles from $H$ as follows. Let $H_0=H$. As long as the graph $H_i$ has average degree $d_i>84$, we find a cycle $C_{i+1}$ in it of length at least $d_i \log^2 n_0$ and define $H_{i+1}=H_i-E(C_{i+1})$. After some finite number (say $l$) of steps we end up with a graph $H'=H_l$ of average degree at most $84$. We need to bound $l$.

Going from $H_i$ of average degree $d_i$ to $H_j$, the first graph of average degree below $d_i/2$, we remove at most $d_in_0/2$ edges using cycles of length at least $\frac{d_i}{2}\log^2 n_0$. So the number of cycles we removed is at most $n_0/\log^2 n_0$. Thus if $H$ had average degree $d$ then $l \le n_0\log_2 d/\log^2 n_0 \le 2n_0/\log n_0$, as needed.
\end{proof}

To conclude this section, we prove Theorem~\ref{thm:main} in the range $\om{ \frac{\log\log n}{n} } \le p \le \frac{\log^{10} n}{n}$ by observing that whp $G(n,p)$ contains no more than $o(n)$ short cycles.

\begin{LEMMA} \label{lem:shortcyc}
Let $p<\frac{\log^{10} n}{n}$. Then whp $G(n,p)$ contains no more than $\sqrt{n}$ cycles of length at most $\log\log n$.
\end{LEMMA}
\begin{proof}
Let $X_k$ be the number of cycles of length $k$ in $G(n,p)$. The number of cycles in $K_n$ of length $k$ is at most $n^k$, and each cycle has probability $p^k$ of being included in $G(n,p)$, hence $\Exp(X_k)\le (np)^k \le (\log^{10} n)^k$.
So if $X=\sum_{k=3}^{\log\log n} X_k$ is the number of cycles of length at most $\log\log n$, then we clearly have
\[ \Exp(X)\le \sum_{k=3}^{\log\log n} \Exp(X_k)\le \sum_{k=3}^{\log\log n} (\log^{10} n)^k \le (\log^{10} n)^{2\log\log n} \]
using $\log^{10} n>2$.

Hence we can apply Markov's inequality to bound the probability that there are more than $\sqrt{n}$ short cycles:
\[ \Prb(X \ge \sqrt{n}) \le \frac{(\log^{10} n)^{2\log\log n}}{\sqrt{n}} = \exp\large\{20(\log\log n)^2 - (\log n)/2\large\} =o(1). \]
\end{proof}

\begin{COR} \label{lem:smalleul}
Let $ p < \frac{\log^{10} n}{n}$. Then whp any Euler subgraph $H$ of $G\sim G(n,p)$ can be decomposed into $o(n)$ cycles.
\end{COR}
\begin{proof}
Use Corollary~\ref{thm:cycrem} to remove $o(n)$ edge-disjoint cycles and end up with a graph $H_1$ of average degree at most 84. Note that $H_1$ is still an Euler graph, so we can break the edges of $G_1$ into cycles arbitrarily. We claim that the number of cycles we get is $o(n)$. Indeed, by Lemma~\ref{lem:shortcyc}, whp there are at most $\sqrt{n}=o(n)$ short cycles, i.e., of length at most $\log\log n$, while the number of long cycles can be bounded by the number of edges divided by $\log\log n$. Since $H_1$ contains a linear number of edges, the number of long cycles is $O(\frac{n}{\log\log n})=o(n)$ and we are done.
\end{proof}

Our theorem for small $p$ is then an immediate consequence of Lemma~\ref{lem:matching2} and Corollary~\ref{lem:smalleul}.

\begin{THM} \label{thm:smallp}
Let $\om{ \frac{\log\log n}{n} } < p < \frac{\log^{10} n}{n}$. Then whp $G\sim G(n,p)$ can be decomposed into $\frac{\odd(G)}{2}+o(n)$ cycles and edges.
\end{THM}
\qed

\section{The main ingredients for the dense case} \label{sec:lemma}

For larger $p$ we use a strong theorem by Broder, Frieze, Suen and Upfal \cite{BFSU} about the existence of edge-disjoint paths in random graphs connecting a prescribed set of vertex pairs. We need the following definition to state it: Suppose $S$ is a set of vertices in a graph $G$. We define the \emph{maximum neighborhood-ratio} function $r_G(S)$ to be $\max_{v\in V(G)} \frac{|N_G(v)\cap S|}{|N_G(v)|}$.

\begin{THM}[Broder-Frieze-Suen-Upfal] \label{thm:pathconnect}
Let $p = \om{ \frac{\log n}{n} }$. Then there are two constants $\alpha,\beta>0$ such that with probability at least $1-\frac{1}{n}$ the following holds in $G\sim G(n,p)$. For any set $F=\{(a_i,b_i) | a_i,b_i\in V, i=1,\ldots,k\}$ of at most $\alpha \frac{n\log (np)}{\log n}$ disjoint pairs in $G$ satisfying the property below, there are vertex-disjoint paths connecting $a_i$ to $b_i$:
\begin{itemize}
  \item There is no vertex $v$ which has more than a $\beta$-fraction of its neighborhood covered by the vertices in $F$. In other words, $r_G(S)\le \beta$, where $S$ is the set of vertices appearing in $F$.
\end{itemize}
\end{THM}

We shall use the following statement to establish this property, so that we can apply Theorem~\ref{thm:pathconnect} in our coming proofs.

\begin{LEMMA} \label{lem:sparsify}
Let $M$ be a matching covering some vertices of $V$, and let $G\sim G(n,\frac{\log^3 n}{n})$ be a random graph on the same vertex set, where $G$ and $M$ are not necessarily independent. Then with probability $1-n^{-\om{1}}$ we can break $G$ into $\log n$ random graphs $G_i\sim G(n,\frac{\log^2 n}{n})$ and $M$ into $\log n$ submatchings $M_i$ of at most $\frac{n}{\log n}$ edges each, such that $r_{G_i}(S_i)\le \frac{4}{\sqrt{\log n}}$ for all $i=1,\ldots,\log n$, where $S_i$ is the set of endvertices of $M_i$.
\end{LEMMA}
\begin{proof}
First, we partition the edge set of $G$ into $\log n$ graphs $G_1, \ldots, G_{\log n}$ by choosing an index $i_e$ for each edge $e$ uniformly and independently at random, and placing $e$ in $G_{i_e}$. Then each $G_i$ has distribution $G(n,\frac{\log^2 n}{n})$. Then we can apply Claim~\ref{lem:chernoff}(b) and (d) to the degree of each vertex of each of the $G_{i}$'s, and use the union bound to see that all the $G_i$ have minimum degree at least $\frac{\log^2 n}{2}$ and maximum degree at most $2\log^2 n$ with probability $1-2n\log n\cdot e^{-\log^2 n/40} = 1-n^{-\om{1}}$.

Now break the $M$ into $\log n$ random matchings $M_1, \ldots, M_{\log n}$ similarly, by placing each edge $f\in M$ independently in $M_{i_f}$ where $i_f$ is a random index chosen uniformly. Since there are at most $n/2$ edges in $M$, another application of Claim~\ref{lem:chernoff}(b) gives that with probability $1-\log ne^{-n/20\log n}=1-n^{-\om{1}}$ each $M_i$ contains at most $\frac{n}{\log n}$ edges.

If $G_i$ has maximum degree at most $2\log^2 n$, then the neighborhood of an arbitrary vertex $v$ in $G_i$ may meet at most $2\log^2 n$ edges from $M$. The probability that at least $(\log n)^{3/2}$ of them are selected in $M_i$ is at most
\[ \binom{2\log^2 n}{\log^{3/2} n}\frac{1}{(\log n)^{\log^{3/2} n}} \le \left( \frac{2e\log^2n}{\log^{3/2}\cdot \log n} \right)^{\log^{3/2} n} \le \left( \frac{2e}{\log^{1/2} n} \right)^{\log^{3/2} n} \le e^{-\log n \cdot \log\log n}. \]
So taking the union bound over all vertices $v$ and indices $i$ gives that with probability $1-n^{-\om{1}}$, the neighborhood of any $v$ in any $G_i$ meets at most $\log^{3/2} n$ of the edges in $M_i$, and thus it contains at most $2\log^{3/2} n$ vertices from $S_i$. Since all neighborhoods have size at least $\frac{\log^2 n}{2}$, we get that $r_{G_i}(S_i)\le \frac{4}{\sqrt{\log n}}$ for all $i$.
\end{proof}

The next theorem is the main tool on our way to the proof of Theorem~\ref{thm:main}.

\begin{THM} \label{thm:decomp}
Let $\frac{2\log^5 n}{n} \le p \le n^{-1/6}$ and $G\sim G(n,p)$. Suppose $G$ is randomly split into $G'$ and $G''$ by placing its edges independently into $G'$ with probability $p'=\frac{2\log^5 n}{np}$ and into $G''$ otherwise. Then whp for any subgraph $H$ of $G''$ such that $G'\cup H$ is Euler, $G'\cup H$ can be decomposed into $o(n)$ cycles.
\end{THM}

\begin{proof}
Note that, although $G'$ and $G''$ are far from being independent, $G'$ on its own has distribution $G(n,p'p)=G(n,\frac{2\log^5 n}{n})$ and $G''$ has distribution $G(n,(1-p')p)$ where $(1-p')p < n^{-1/6}$. It is easy to see from Claim~\ref{lem:chernoff} that whp $G$ has maximum degree at most $2n^{5/6}$.

By Corollary~\ref{thm:cycrem}, whp any $H\subs G''$ can be decomposed into $o(n)$ cycles and a graph $H_0$ of average degree at most $84$. Therefore it is enough for us to show that whp $G'$ satisfies the following: for any $H_0$ containing at most $42n$ edges such that $G'\cup H_0$ is Euler, $G'\cup H_0$ is an edge-disjoint union of $o(n)$ cycles. Our plan is to break $H_0$ into few matchings and then to use Theorem~\ref{thm:pathconnect} on random subgraphs of $G'$ to connect them into cycles.

So define $V_0\subs V$ to be the set of vertices of degree at least $\frac{\log^2 n}{2}$ in $H_0$, and let $V_1=V-V_0$ be the rest. Note that $|V_0|=O(\frac{n}{\log^2 n})$. We break into matchings in two rounds: first we take care of the edges spanned by $V_1$, then the ones crossing between $V_0$ and $V_1$. Let us split $G'$ into two random graphs $G_1,G_2 \sim G(n,\frac{\log^5 n}{n})$ by placing each edge of $G'$ independently in one of them with probability $1/2$.

Now let us consider the subgraph of $H_0$ spanned by $V_1$: the maximum degree is at most $\frac{\log^2 n}{2} - 1$, so we can break the edge set into $\log^2 n$ matchings $M_1 , \ldots , M_{\log^2 n}$. To find the cycles, we also split $G_1$ into $\log^2 n$ parts $G_{1,1}, \ldots, G_{1,\log^2 n}$ so that each $G_{1,i}$ has distribution $G(n,\frac{\log^3 n}{n})$. We can use Lemma~\ref{lem:sparsify} to further break each $M_i$ and $G_{1,i}$ into $\log n$ parts $M_{i,j}$ and $G_{1,i,j}$, respectively, so that whp each $M_{i,j}$ contains $O(n/\log n)$ edges, and the endvertices of $M_{i,j}$ only cover a $o(1)$-fraction of the neighborhood of any vertex in $G_{1,i,j}$.

We create the cycles as follows: if $M_{i,j}$ consists of the edges $v_1v'_1, \ldots, v_kv'_k$, then we choose the corresponding set of pairs to be $F_{1,i,j}=\{(v'_1,v_2), (v'_2,v_3), \ldots, (v'_k,v_1)\}$. Here the above properties of $M_{i,j}$ and $G_{1,i,j}$ ensure that we can apply Theorem~\ref{thm:pathconnect}, and with probability at least $1-\frac{1}{n}$ the matching can be closed into a cycle. Hence with probability $1-\frac{\log^3 n}{n}=1-o(1)$ all the $M_{i,j}$'s can be covered by $\log^3 n$ cycles altogether.

Let us turn to the edges of $H_0$ between $V_0$ and $V_1$, and define the following auxiliary multigraph $G_a$ on $V_1$: for each $v\in V_0$, pair up its neighbors in $V_1$ (except maybe one vertex, if the neighborhood has odd size) and let $E_v$ be the set of edges -- a matching -- corresponding to this pairing. Define the edge set of $G_a$ to be the disjoint union of the $E_v$ for $v\in V_0$. The idea is that an edge $ww' \in E_v$ corresponds to the path $wvw'$, so we want to find cycles covering the edges in $G_a$ and then lead them through the original $V_0-V_1$ edges.

By the definition of $V_1$, the maximum degree in $G_a$ is at most $\frac{\log^2 n}{2}-1$, so the edge set $\cup_{v\in V_0} E_v$ can be split into $\log^2 n$ matchings $N_1, \ldots, N_{\log^2 n}$. Now it is time to break $G_2$ into $\log^2 n$ random subgraphs $G_{2,i}$ of distribution $G(n,\frac{\log^3 n}{n})$ each. Once again, we use Lemma~\ref{lem:sparsify} to prepare for the cycle cover by splitting each of the $N_i$ and $G_{2,i}$ into $\log n$ parts $N_{i,j}$ and $G_{2,i,j}$. When we define the set of pairs $F_{2,i,j}$, we need to be a little bit careful: we must make sure that no cycle contains more than one edge from any given $E_v$. This way the cycles do not become self-intersecting after switching the edges from $E_v$ back to the corresponding paths through $v$. Since the maximum degree of $G$, and hence the cardinality of $E_v$, is at most $2n^{5/6}$, we may achieve this by using at most $2n^{5/6}$ cycles per matching. Indeed, split $N_{i,j}$ into at most $2n^{5/6}$ subsets $N_{i,j,k}$ so that none of the $N_{i,j,k}$ contains more than one edge from the same $E_v$ (this can be done greedily). Then define the sets of pairs $F_{2,i,j,k}$ for $k=1,\dots, 2n^{5/6}$ to close $N_{i,j,k}$ into a cycle the same way as before, and take $F_{2,i,j}=\cup_{k=1}^{2n^{5/6}} F_{2,i,j,k}$.

As above, all conditions of Theorem~\ref{thm:pathconnect} are satisfied when we use $G_{2,i,j}$ to find the paths corresponding to $F_{2,i,j}$ that close $N_{i,j}$ into cycles, and since the error probabilities were all $O(\frac{1}{n})$, whp we can simultaneously do so for all $i$ and $j$. We have $\log^3 n$ matchings, so in total we get $2n^{5/6}\log^3 n=o(n)$ edge-disjoint cycles that cover all but $o(n)$ edges of $H_0$ between $V_0$ and $V_1$ (missing at most one incident edge for each $v\in V_0$).

Finally, we apply Corollary~\ref{thm:cycrem} on the subgraph of $H_0$ induced by $V_0$ to see that the edges spanned by $V_0$ can be partitioned into $O(n/\log^3 n)$ cycles and $O(n/\log^2 n)=o(n)$ edges (recall that $|V_0|=O(n/\log^2 n)$).

So far we have found $o(n)$ edge-disjoint cycles in $G'\cup H$. Once we remove them, we get an Euler graph containing only $o(n)$ edges from $H$. So we can find $o(n)$ edge-disjoint cycles covering all of them and remove these cycles, as well, to get an Euler subgraph of $G' \sim G(n,\frac{\log^5 n}{n})$. Now Corollary~\ref{lem:smalleul} shows that we can partition the remaining graph into $o(n)$ cycles, concluding our proof.
\end{proof}

\section{Cycle-edge decompositions in dense random graphs} \label{sec:dense}

At last, we are ready to prove Theorem~\ref{thm:main} in the denser settings. The case $p\le n^{-1/6}$ is fairly straightforward from our previous results, we just need to be a little bit careful.

\begin{THM} \label{thm:mediump}
Let $\frac{\log^6 n}{n}\le p \le n^{-1/6}$. Then whp $G\sim G(n,p)$ can be decomposed into $\frac{\odd(G)}{2}+o(n)$ cycles and edges.
\end{THM}
\begin{proof}
We split $G$ into the union of three disjoint random graphs $G_1, G_2$ and $G_3$ by putting each edge $e\in E(G)$ independently into one of the graphs. With probability $p_1=\frac{2\log^5 n}{np}$ we place $e$ into $G_1$, with probability $p_2=\frac{\log^2 n}{np}$ we place it into $G_2$, and with probability $1-p_1-p_2$ we place it into $G_3$. This way $G_1\sim G(n,\frac{2\log^5 n}{n})$ and $G_2\sim G(n,\frac{\log^2 n}{n})$.

Now let $S$ be the set of odd-degree vertices in $G$. Applying Lemma~\ref{lem:matching} to $S$ in $G_2$ gives a set $E_0$ of $|S|/2 + o(n)$ edges in $G_2$ such that $G-E_0$ is Euler. Taking $H$ to be the subgraph $G_2\cup G_3 - E_0$ of $G''=G_2\cup G_3$ and setting $G'=G_1$, we can apply Theorem~\ref{thm:decomp} to split the edge set of $G-E_0$ into $o(n)$ cycles. The theorem follows.
\end{proof}

To get a tight result for larger $p$, we must remove cycles containing nearly all vertices. A recent result by Knox, K\"uhn and Osthus \cite{KKO} helps us to find many edge-disjoint Hamilton cycles in random graphs.

\begin{THM}[Knox-K\"uhn-Osthus] \label{thm:hamilton}
Let $\frac{\log^{50} n}{n} \le p \le 1-n^{-1/5}$ and $G\sim G(n,p)$. Then whp $G$ contains $\lfloor \delta(G)/2 \rfloor$ edge-disjoint Hamilton cycles.
\end{THM}

Let us point out that we do not actually need such a strong result: $\delta(G)/2 - n^{\eps}$ disjoint Hamilton cycles would also suffice for our purposes.

\begin{THM} \label{thm:largep}
Let $p\ge n^{-1/6}$. Then whp $G\sim G(n,p)$ can be decomposed into $\frac{\odd(G)}{2}+\frac{np}{2}+o(n)$ edges.
\end{THM}
\begin{proof}
Similarly to Theorem~\ref{thm:mediump}, we partition $G$ into the union of four disjoint random graphs $G_1, G_2, G_3$ and $G_4$ by assigning each edge of $G$ independently to $G_1$ with probability $p_1=\frac{2\log^5 n}{np}$, to $G_2$ with probability $p_2=\frac{n^{4/5}\log^3 n}{np}$, to $G_3$ with probability $p_3=\frac{\log^2 n}{np}$ and to $G_4$ otherwise (with probability $p_4=1-p_1-p_2-p_3=1-o(1)$). It is easy to see from the Chernoff bound (Claim~\ref{lem:chernoff}(a)) that whp the maximum degree of $G_3$ is at most $npp_3+n^{3/5}\le 2n^{3/5}$, and the maximum degree of $G_4$ is at most $npp_4+n^{3/5}$. Let us assume this is the case.

Let $S$ be the set of odd-degree vertices in $G$. Now as before, we use Lemma~\ref{lem:matching} to find a set $E_0$ of $\frac{|S|}{2}+o(n)$ edges in $G_3$ such that $G-E_0$ is Euler. Notice that $G_4\sim G(n,pp_4)$, where $pp_4=p(1-o(1))$, so whp $G_4$ has minimum degree at least $npp_4-n^{3/5}$ by Claim~\ref{lem:chernoff}(c). Also, $pp_4<1-n^{-1/5}$ (because $pp_2>n^{-1/5}$), hence we can apply Theorem~\ref{thm:hamilton} and find $\frac{npp_4-n^{3/5}}{2}$ edge-disjoint Hamilton cycles in it. Let $H_0$ be the graph obtained from $G_3\cup G_4 - E_0$ by removing these cycles. Then the maximum degree of $H_0$ is at most $4n^{3/5}$, hence we can break the edge set of $H_0$ into $n^{4/5}$ matchings $M_i$.

We want to use Theorem~\ref{thm:pathconnect} to close the $M_i$'s into cycles, so let us split $G_2$ into $n^{4/5}$ random graphs $G'_i \sim G(n,\frac{\log^3 n}{n})$ uniformly. By Lemma~\ref{lem:sparsify} we can further partition each $M_i$ and $G'_i$, with probability $1-n^{-\om{1}}$, into $\log n$ matchings and graphs, $M_{i,j}$ and $G'_{i,j}$ in such a way that we can apply Theorem~\ref{thm:pathconnect} on $G'_{i,j}$ for any pairing $F_{i,j}$ on the vertices of $M_{i,j}$. Choose $F_{i,j}$, as before, so that the resulting paths together with $M_{i,j}$ form a cycle. Then with probability at least $1-\frac{1}{n}$ the theorem produces the required cycle, so whp all $n^{4/5}\log n$ cycles exist simultaneously.

This way we find $n^{4/5}\log n=o(n)$ edge-disjoint cycles covering $H_0$ and some edges in $G_2$. Let $H$ be the graph containing the unused edges of $G_2$. Then $G_1\cup H$ is Euler, and we can apply Theorem~\ref{thm:decomp} with the host graph $G_1 \cup G_2$ from distribution $G(n,p(p_1+p_2))$, and the partition $G'=G_1$, $G''=G_2$. This gives us a decomposition of $G'\cup H$ into $o(n)$ cycles whp, completing our proof.
\end{proof}

\section{Concluding remarks}

The above proof settles the question for $p=\om{\frac{\log\log n}{n}}$, but it would be nice to have a result for the whole probability range. The bottleneck in our proof is Lemma~\ref{lem:matching2}, where we obtain a small edge set $E_0$ such that $G(n,p)-E_0$ is Euler. We believe that similar ideas can be applied to prove this lemma for even smaller values of $p$ if one puts more effort into finding short paths between vertices not covered by the matching. In any case, it seems that the asymptotics of the optimum is defined by the smallest such $E_0$ for any $p\le \log n/n$, so a complete solution to the problem would first need to describe this minimum in the whole range.

Another direction might be to further explore the error term of our theorem. One clear obstacle to an improvement using our methods is Corollary~\ref{lem:smalleul}. While we could slightly improve it to give a $O(n/\log n)$ bound, showing that the error term is significantly smaller would need more ideas.

\bigskip

\noindent{\bf Acknowledgement.}
The authors would like to thank the anonymous referee for helpful remarks. They are also grateful to D. Conlon and J. Fox for stimulating discussions of the problem treated in this paper.

\end{document}